\documentclass[a4paper, 11pt]{amsart}
\usepackage{latexsym}
\usepackage{tikz-cd}

\usepackage{xcolor}
\usepackage[normalem]{ulem}

\def\Z{\mathbb{Z}}

\def\Q{\mathrm{Q}}

\def\H{\mathrm{H}}

\def\E{\mathcal{E}}

\usepackage{mathtools}

\usepackage{geometry}
\newtheorem{prop}{Proposition}[section]
\newtheorem{lem}[prop]{Lemma}

\newtheorem{thm}[prop]{Theorem}

\theoremstyle{definition}
\newtheorem{rem}[prop]{Remark}
\newtheorem{defi}[prop]{Definition}
\newtheorem{ex}[prop]{Example}

\begin{document}

\title[Highly symmetric Archdeacon embeddings]{A class of highly symmetric\\ Archdeacon embeddings}

\author[S. Costa]{Simone Costa}
\address{DICATAM - Sez. Matematica, Universit\`a degli Studi di Brescia, Via
Branze 43, I-25123 Brescia, Italy}
\email{simone.costa@unibs.it}
\author[L. Mella]{Lorenzo Mella}
\address{Dip. di Scienze Fisiche, Informatiche, Matematiche, Universit\`a degli Studi di Modena e Reggio Emilia, Via Campi 213/A, I-41125 Modena, Italy}
\email{lorenzo.mella@unipr.it}

\begin{abstract}
Archdeacon, in his seminal paper \cite{A}, defined the concept of Heffter array to provide explicit constructions of biembeddings of the complete graph $K_v$ into orientable surfaces, the so-called Archdeacon embeddings, and proved that these embeddings are $\mathbb{Z}_{v}$-regular.

In this paper, we show that an Archdeacon embedding may admit an automorphism group that is strictly larger than $\mathbb{Z}_{v}$. Indeed, as an application of the interesting class of arrays recently introduced by Buratti in \cite{B}, we exhibit, for infinitely many values of $v$, an embedding of this type having full automorphism group of size ${v \choose 2}$ that is the largest possible one.
\end{abstract}

\keywords{Heffter array, Archdeacon Embedding, Automorphism Group}
\subjclass[2010]{05B20, 05C10, 05C60}

\maketitle

\section{Introduction}
An $m\times n$ partially filled (p.f., for short) array on a set $\Omega$ is an $m \times n$ matrix whose elements belong to $\Omega$
and where some cells can be empty. In 2015, Archdeacon (see \cite {A}), introduced a class of p.f. arrays that have been extensively studied:
the \emph{Heffter arrays}.
\begin{defi}\label{def:H}
A \emph{Heffter array} $\H(m,n; h,k)$ is an $m \times n$ p. f. array with entries in $\Z_{2nk+1}$ such that:
\begin{itemize}
\item[(\rm{a})] each row contains $h$ filled cells and each column contains $k$ filled cells,
\item[(\rm{b})] for every $x\in \Z_{2nk+1}\setminus\{0\}$, either $x$ or $-x$ appears in the array,
\item[(\rm{c})] the elements in every row and column sum to $0$ (in $\Z_{2nk+1}$).
\end{itemize}
\end{defi}
These arrays have been introduced because of their vast variety of applications and links to other problems and concepts, such as orthogonal cycle decompositions and $2$-colorable embeddings (briefly \emph{biembeddings}), see for instance \cite{A, CDY, CMPPHeffter, DM}. The existence problem of Heffter arrays has been also deeply investigated: we refer to the survey \cite{DP} for the known results in this direction.
This paper will focus mainly on the connection between p.f. arrays and embeddings. To explain this link, we first recall some basic definitions, see \cite{Moh, MT}.
\begin{defi}
Given a graph $\Gamma$ and a surface $\Sigma$, an \emph{embedding} of $\Gamma$ in $\Sigma$ is a continuous injective mapping $\psi: \Gamma \rightarrow \Sigma$, where $\Gamma$ is viewed with the usual topology as a $1$-dimensional simplicial complex.
\end{defi}
The connected components of $\Sigma \setminus \psi(\Gamma)$ are called $\psi$-\emph{faces}. Also, with a slight abuse of notation, we say that a circuit $F$ of $\Gamma$ is a face (induced by the embedding $\psi$) if $\psi(F)$ is the boundary of a $\psi$-face. Then, if each $\psi$-face is homeomorphic to an open disc, the embedding $\psi$ is said to be \emph{cellular}.
In this context, we say that two embeddings $\psi: \Gamma \rightarrow \Sigma$ and $\psi': \Gamma' \rightarrow \Sigma'$ are \emph{isomorphic} whenever there exists a graph isomorphism $\sigma: \Gamma\rightarrow \Gamma'$ such that $\sigma(F)$ is a $\psi'$-face if and only if $F$ is a $\psi$-face. Here we say that $\sigma$ is an embedding isomorphism or, if $\psi=\psi'$, an embedding automorphism.

Archdeacon, in his seminal paper \cite{A}, showed that, if some additional technical conditions are satisfied, Heffter arrays provide explicit constructions of $\mathbb{Z}_{v}$-regular biembeddings of complete graphs $K_v$ into orientable surfaces. Here \emph{$\mathbb{Z}_{v}$-regular} means that $\mathbb{Z}_v$ acts sharply regularly on the vertex set, hence $\mathbb{Z}_v$ is contained in the automorphism group of $\psi$, denoted by $Aut(\psi)$.
Following \cite{Costa}, the embeddings defined using this construction via p. f. arrays will be denoted as \emph{embeddings of Archdeacon type} or, more simply, \emph{Archdeacon embeddings}.

Indeed, this kind of embedding can be considered also for variations of the concept of Heffter arrays, such as the \emph{non-zero sum Heffter arrays} discussed in \cite{CostaDellaFiorePasotti, PM, MT1} (see also \cite{DP} for other variations and generalizations). In \cite{Costa}, the author provided a generalization of both the \emph{Heffter} and the \emph{non-zero sum Heffter} arrays and showed that the Archdeacon embedding can be also defined in this more general context. Given $v=2nk+1$ and a group $G$ of order $v$, he defined the so-called \emph{quasi}-\emph{Heffter array $A$ over $G$}\footnote{In \cite{Costa} it was defined the concept of quasi-Heffter only in the case $G=\Z_v$ but, following \cite{CPPBiembeddings}, we provide the definition for generic groups.} (denoted as $\Q\H(m,n; h,k)$) by considering an $m\times n$ p.f. array with elements in $G$ such that:
\begin{itemize}
\item[($\rm{a_1})$] each row contains $h$ filled cells and each column contains $k$ filled cells,
\item[($\rm{b_1})$] the multiset $[\pm x \mid x \in A]$ contains each element of $G\setminus \{0\}$ exactly once.
\end{itemize}
If every row and every column of a quasi-Heffter sum to zero, consistently with the classical Heffter arrays nomenclature, we denote this array by $\H(m,n; h,k)$ over $G$ and, if also $m=k$ and $n=h$, by $\H(m,n)$ over $G$.

Using such arrays, the Archdeacon embedding into orientable surfaces is still well defined, with essentially the same construction of \cite{A}.
Then, again in \cite{Costa}, it was proved that for an embedding constructed from a $\Q\H(m,n; h,k)$ over $\Z_v$, its full automorphism group, which acts sharply transitively on the vertex set, is \textbf{almost always} exactly $\mathbb{Z}_{v}$. Note that several examples of $\Z_v$-regular embeddings whose full automorphism group is larger than $\Z_v$ are known (see for instance Example 3.4 of \cite{Costa} or \cite{CFP}), but none of them arises from the Archdeacon construction. For this reason, we find it natural to investigate this existence problem. The main result here presented is indeed the existence of an infinite family of such embeddings whose automorphism groups are strictly larger than $\Z_v$.

In Section $2$ we will provide a formal definition of the Archdeacon embedding, starting from a Heffter array over a group $G$. Then, in Section 3 we will analyze one of the interesting classes of Heffter arrays over $\mathbb{F}_q$ recently introduced by Buratti in \cite{B}. We will show that using these arrays it is possible to obtain infinitely many embeddings of Archdeacon type with a rich automorphism group: more precisely the stabilizer of a given point in the automorphism group has size $(v-1)/2$ and we will show that this is the largest possible one for such embeddings. Moreover, since $\mathbb{F}_p=\Z_p$ (assuming $p$ is a simple prime), infinitely many of these embeddings arise from classical Heffter arrays.
\section{The Archdeacon Embedding and its Automorphisms}
Following \cite{GG,GT, JS}, we provide an equivalent, but purely combinatorial, definition of a graph embedding into a surface.
Here we denote by $D(\Gamma)$ the set of all the oriented edges of the graph $\Gamma$ and, given a vertex $x$ of $\Gamma$, $N(\Gamma,x)$ denotes the set of vertices adjacent to $x$ in $\Gamma$.
\begin{defi}\label{DefEmbeddings}
Let $\Gamma$ be a connected multigraph. A \emph{combinatorial embedding} of $\Gamma$ (into an orientable surface) is a pair $\Pi=(\Gamma,\rho)$ where $\rho: D(\Gamma)\rightarrow D(\Gamma)$ satisfies the following properties:
\begin{itemize}
\item[(a)] for any $y\in N(\Gamma,x)$, there exists $y'\in N(\Gamma,x)$ such that $\rho(x,y)=(x,y')$,
\item[(b)] we define $\rho_x$ as the permutation of $N(\Gamma,x)$ such that, given $y\in N(\Gamma,x)$, $\rho(x,y)=(x,\rho_x(y))$. Then the permutation $\rho_x$ is a cycle of order $|N(\Gamma,x)|$.
\end{itemize}
If properties $(a)$ and $(b)$ hold, the map $\rho$ is said to be a \emph{rotation} of $\Gamma$.
\end{defi}
Then, as reported in \cite{GG}, a combinatorial embedding $\Pi=(\Gamma,\rho)$ is equivalent to a cellular embedding $\psi$ of $\Gamma$ into an orientable surface $\Sigma$ (see also \cite{A}, Theorem 3.1).

Now we recall the definition of Archdeacon embedding. Here we report it in the case of Heffter arrays over a group $G$, but we remark that the same embedding is still well-defined also for quasi-Heffter arrays (see \cite{Costa}, Definition 2.4). We first introduce some notation.
The rows and the columns of an $m\times n$ array $A$ are denoted by $R_1,\ldots, R_m$ and by $C_1,\ldots, C_n$, respectively. Also, by $\E(A)$, $\E(R_i)$, $\E(C_j)$ we mean the list of the elements of the filled cells of $A$, of the $i$-th row and of the $j$-th column, respectively. Given an $m\times n$ p.f. array $A$, by $\omega_{R_i}$ and $\omega_{C_j}$ we denote respectively an ordering of $\E(R_i)$ and $\E(C_j)$, and we define by $\omega_r=\omega_{R_1}\circ \cdots \circ \omega_{R_m}$ the ordering for the rows and by $\omega_c=\omega_{C_1}\circ \cdots \circ \omega_{C_n}$ the ordering for the columns.
\begin{defi}\label{Compatible}
Given a Heffter array $A$, the orderings $\omega_r$ and $\omega_c$ are said to be \emph{compatible} if $\omega_c \circ \omega_r$ is a cycle of order $|\E(A)|$.
\end{defi}

Now we are ready to recall the definition of Archdeacon embedding, see \cite{A}.
\begin{defi}\label{ArchdeaconEmbedding}
Let $A$ be an $\H(m,n;h,k)$ over the group $G$ that admits two compatible orderings $\omega_r$ and $\omega_c$. Let $\rho_0$ be the following permutation on
$\pm \E(A)=G\setminus \{0\}$:
\begin{eqnarray}\label{ArchEmb}
\rho_0(a)&=&\begin{cases}
-\omega_r(a)\mbox{ if } a\in \E(A),\\
\omega_c(-a)\mbox{ if } a\in -\E(A).\\
\end{cases}
\end{eqnarray}
Let $\rho$ be the map on the set of the oriented edges of $K_v$ defined as follows
\begin{eqnarray}\label{ArchRho}
\rho((x,x+a))&=& (x,x+\rho_0(a)).
\end{eqnarray}
Then, the pair $\Pi = (K_v, \rho)$ is said to be an \textit{Archdeacon embedding}.
\end{defi}
Indeed, Archdeacon proved in \cite{A} that, since the orderings $\omega_r$ and $\omega_c$ are compatible, the map $\rho$ is a rotation of $K_{v}$ (he considered the case $G=\Z_v$ but, as noted in \cite{CPPBiembeddings}, the same proof holds in general), and thus the pair $(K_v,\rho)$ is a combinatorial embedding of $K_v$. More precisely he showed the following theorem.

\begin{thm}\label{HeffterBiemb} Let $A$ be an $\H(m,n;h,k)$ over the group $G$ that admits two compatible orderings $\omega_r$ and $\omega_c$. Then there exists a biembedding $\Pi$ of $K_{2nk+1}$, such that every edge is on a face whose boundary length is $h$ and on a face whose boundary length is $k$.

Moreover, $\Pi$ admits $G$ as a sharply transitive automorphism group.
\end{thm}
He also described the faces induced by the Archdeacon embedding under the conditions of Theorem \ref{HeffterBiemb}.

For this purpose, we take a p.f. array $A$ that is an $\H(m,n;h,k)$ admitting two compatible orderings $\omega_r$ and $\omega_c$. We consider the oriented edge $(x,x+a)$ with $a \in \E(A)$. Due to Theorem 1.1 of \cite{A}, the directed edge $(x,x+a)$ belongs to the face $F_1$ whose boundary is
\begin{equation}\label{F1}\left(x,x+a,x+a+\omega_c(a),\ldots,x+\sum_{i=0}^{k-1} \omega_c^i(a)\right).\end{equation}
Let us now consider the oriented edge $(x,x+a)$ with $a\not \in \E(A)$.
In this case, $(x,x+a)$ belongs to the face $F_2$ whose boundary is
\begin{equation}\label{F2}\left(x,x+\sum_{i=1}^{h-1}\omega_{r}^{-i}(-a),x+\sum_{i=1}^{h-2}\omega_{r}^{-i}(-a),
\dots,x+\omega_{r}^{-1}(-a)\right).\end{equation}
A priori these faces are circuits but, under suitable conditions, we can prove that they are simple cycles. To be more precise we need to introduce some further definitions.

Given a finite subset $T$ of an abelian group $G$ and an ordering $\omega = (t_1, t_2,\dots, t_k)$ of the elements in $T$,  for any $i\in [1, k]$ let $s_i =\sum_{j=1}^i t_j$ be the $i$-th partial sum of $\omega$. The ordering $\omega$ is said to be \emph{simple}\ if $s_b\not=s_c$ for any $1\leq b,c\leq k$.
Given an $m\times n$ p.f. array $A$ whose entries belong to a given group $G$, and given an ordering $\omega_{C_i}$ for any column $C_i$ where $1\leq i\leq n$ and an ordering $\omega_{R_j}$ for any row $R_j$ where $1\leq j\leq m$, we say that $\omega_r=\omega_{R_1}\circ \cdots\circ\omega_{R_m}$ and $\omega_c=\omega_{C_1}\circ \cdots\circ\omega_{C_n}$ are \emph{simple} if each $\omega_{C_i}$ and $\omega_{R_j}$ is simple. If the natural orderings, from top to bottom for each column and from left to right for each row are simple we say that the array is \emph{globally simple}. We can restate the main result of Archdeacon \cite{A} as follows:
\begin{thm}\label{HeffterBiemb} Let $A$ be an $\H(m,n;h,k)$ over the group $G$ that admits two compatible and simple orderings $\omega_r$ and $\omega_c$. Then there exists a biembedding $\Pi$ of $K_{2nk+1}$ such that every edge is on a face whose boundary is an $h$-cycle (i.e. on a simple face whose length is $h$) and on a face whose boundary is a $k$-cycle (i.e. on a simple face whose length is $k$).

Moreover, $\Pi$ admits $G$ as a sharply transitive automorphism group.
\end{thm}

Now, in order to investigate the automorphism group of an embedding of Archdeacon type, we revisit the necessary conditions stated in \cite{Costa}.

We first recall that also the notions of embedding isomorphism and automorphism can be defined purely combinatorially as follows (see Korzhik and Voss \cite{Korzhik}, page 61).
\begin{defi}\label{DefEmbeddingsIs}
Let $\Pi:= (\Gamma,\rho)$ and $\Pi':= (\Gamma',\rho')$ be two combinatorial embeddings of $\Gamma$ and $\Gamma'$, respectively. We say that $\Pi$ is \emph{isomorphic} to $\Pi'$ if there exists a graph isomorphism $\sigma: \Gamma\rightarrow \Gamma'$ such that, for any $(x,y)\in D(\Gamma)$, we have either
\begin{equation}\label{eq11}
\sigma\circ \rho(x,y)=\rho'\circ \sigma(x,y)
\end{equation}
or
\begin{equation}\label{eq12}
\sigma\circ \rho(x,y)=(\rho')^{-1}\circ \sigma(x,y).
\end{equation}
We also say, with abuse of notation, that $\sigma$ is an \emph{embedding isomorphism} between $\Pi$ and $\Pi'$.
Moreover, if equation (\ref{eq11}) holds, $\sigma$ is said to be an \emph{orientation preserving isomorphism} while,
if (\ref{eq12}) holds, $\sigma$ is said to be an \emph{orientation reversing isomorphism}.
\end{defi}

Let $\Pi$ be an Archdeacon embedding of $K_{v}$ obtained from a Heffter array $A$. Then, we denote by $Aut(\Pi)$ the group of all automorphisms of $\Pi$, and by $Aut^+(\Pi)$ we mean its subgroup of orientation-preserving automorphisms. Similarly, by $Aut_0(\Pi)$ and $Aut_0^+(\Pi)$ we respectively denote the subgroups of $Aut(\Pi)$ and of $Aut^+(\Pi)$ of automorphisms that fix $0$.

Here we have that $Aut^+(\Pi)$ and $Aut_0^+(\Pi)$ are, respectively, normal subgroups of $Aut(\Pi)$ and $Aut_0(\Pi)$ (their index is either $1$ or $2$). Moreover, if for every $g \in G$ we denote by $\tau_g$ the translation action by $g$, i.e. $\tau_g(x) = x+g$, it has been proved by Archdeacon (see \cite{A}, Theorem 1.1) that $\tau_g \in Aut^+(\Pi)$, hence $G$ is a subgroup of $Aut^+(\Pi)$.
Thus, the automorphism group of an embedding of Archdeacon type is never trivial. On the other hand, in the case $G=\Z_v$ it has been proven in \cite{Costa} that, if we start from a quasi-Heffer array, $Aut(\Pi)$ is almost always exactly $\Z_v$. Hence we find it natural to investigate whether, given a biembedding of Archdeacon type $\Pi$, the only automorphism of $\Pi$ fixing $0$ must be the identity.

Now, we recall some necessary conditions (see \cite{Costa}) that have to be satisfied by an automorphism of $\Pi$ that fixes zero.
\begin{rem} \label{rem:group}
Let $\Pi=(K_v,\rho)$ be a biembedding of Archdeacon type, where $v=2nk+1$. Then, $\sigma \in Aut_0(\Pi)$ (resp. $Aut_0^+(\Pi)$) acts on $G\setminus \{0\}$ as an element of the dihedral group $\mathrm{Dih}_{2nk}$ (resp. the cyclic group $\Z_{2nk}$).

More precisely, set $\rho_0=(x_1,x_2,\dots,x_{2nk})$, and by reading the indices modulo $2nk$, we have that:
\begin{itemize}
\item given $\sigma\in Aut_0^+(\Pi)$,
$$\sigma|_{K_v\setminus\{0\}}=\rho_0^\ell\mbox{ for some }\ell \in \{1,\dots,2nk\};$$
\item given $\sigma\in Aut_0^-(\Pi)$,$$\sigma(x_j)=x_{\ell-j}\mbox{ for some }\ell \in \{1,\dots,2nk\}.$$
\end{itemize}
\end{rem}

\section{A class of highly symmetric embeddings}
In this section we provide for infinitely many values of $q$ (that here is a prime power) an embedding of Archdeacon kind, defined over the additive group of $\mathbb{F}_q$ (the field of order $q$), whose full automorphism group is of size ${q\choose 2}$: more precisely this embedding admits $\mathbb{F}_q$ as a regular automorphism group and the stabilizer of $0$ is isomorphic to $\Z_{(q-1)/2}$.

We begin by recalling some definitions introduced in \cite{B}.
Given a prime power of the form $q=2mn+1$, an $H(m,n)$ over $\mathbb{F}_q$ is \textit{rank-one} if every row is a multiple of a suitable non-zero vector $X = (x_1, \dotsc, x_n)$ of $\mathbb{F}_q^n$. Note that this implies that every column of the array is a multiple of a suitable non-zero vector $Y = (y_1,\dotsc,y_m)$ of $\mathbb{F}_q^m$.

In the same paper, Buratti provided a construction of rank-one Heffter arrays over $\mathbb{F}_q$, where $q=2mn+1$ is a prime power in several cases: we consider here his constructions when $m$ and $n$ are odd coprime integers.

In this case, given $\xi$ of order $n$ and $\epsilon$ of order $m$ in $\mathbb{F}_q^*=\mathbb{F}_q\setminus \{0\}$, after considering the vectors $X$ and $Y$
\[
\begin{aligned}
X &= (1,\xi,\xi^2,\xi^3, \dotsc, \xi^{n-1}), \\
Y &= (1,\epsilon,\epsilon^2,\epsilon^3, \dotsc, \epsilon^{m-1}),\\
\end{aligned}
\]
and the $m\times n$ array $A_{m,n}=(a_{i,j})$ whose $(i,j)$-cell is $a_{i,j}:=\epsilon^{i-1}\xi^{j-1}$, he proved the following (see \cite{B}, Theorem 5.2):
\begin{thm}\label{lem:glob_simple}\label{B1}
Let $q = 2mn+1$ be a prime power with $m,n$ odd and coprime. Then the array $A_{m,n}$ is a rank-one globally simple $H(m,n)$ over $\mathbb{F}_q$.
\end{thm}

Buratti also noticed that these arrays have many symmetries (the so-called automorphisms of the array). Our goal is now to consider a family of embeddings that arise from these arrays and to prove that also these embeddings have a lot of symmetries. In particular, we will see that the full automorphism groups of these embeddings are the largest possible.

First of all, in order to define the Archdeacon embeddings induced by such arrays, we need to recall the following result of \cite{DM} (see also \cite{CDP}) about compatible orderings of totally filled $m\times n$ arrays.
\begin{prop}\label{prop:orderings}
Let $A$ be a totally filled $m \times n$ array, where at least one between $m$ and $n$ is odd, and let $\ell$ be such that $m-2\ell$ and $n$ are coprime. Let $\omega_{C_i}$ be the natural ordering from top to bottom for every $i \in \{1,\dotsc, n\}$, and let $\omega_{R_j}$ be the natural ordering from left to right for $j \in \{1, \dotsc, m-\ell\}$, and from right to left otherwise. Then the orderings $\omega_c$ and $\omega_r$ are compatible.
\end{prop}
This means that, if $m$ and $n$ are coprime, we can choose $\ell=0$, obtaining that:
\begin{thm}\label{thm:pos_cyclic}
Let $q= 2mn+1$ be a prime power, with $m,n$ odd and coprime. Let $\omega_c$ and $\omega_r$ be the natural orderings from top to bottom and from left to right, and let $A_{m,n}$ be the array defined in Theorem \ref{B1}.

Then $\omega_r$ and $\omega_c$ are simple and compatible orderings of $A_{m,n}$, and they induce an Archdeacon embedding $\Pi_{m,n}$ whose faces are simple cycles of length $m$ and $n$.
\end{thm}
Now we want to investigate the automorphism group of such an embedding.
We begin by considering the orientation-preserving automorphisms.
\begin{prop}
Let $q= 2mn+1$ be a prime power, with $m,n$ odd and coprime, and let $\Pi_{m,n}$ be the Archdeacon automorphism defined in Theorem \ref{thm:pos_cyclic}. Then we have that:
$$Aut_0^+(\Pi_{m,n}) \cong \mathbb{Z}_{mn}.$$
\end{prop}
\begin{proof}
We recall that the array $A_{m,n}$ has, in position $(i,j)$, the element $\epsilon^{i-1}\xi^{j-1}$,  where $\xi$ and $\epsilon$ have respectively order $n$ and $m$ in $\mathbb{F}_q^*=\mathbb{F}_q\setminus \{0\}$. Now, given $1\leq i'\leq m$ and $1\leq j'\leq n$, we define the map
$\lambda_{i',j'}: \mathbb{F}_q\rightarrow \mathbb{F}_q$ such that $\lambda_{i',j'}(z)=a_{i',j'} z=(\epsilon^{i'-1}\xi^{j'-1}) z$. Clearly, this map is a graph automorphism of $K_q$ that fixes $0$. In the following, we will denote for simplicity $a_{i',j'}$ by $\eta$.

We want to prove that $\lambda_{i',j'}$ is also an orientation preserving automorphism of $\Pi_{m,n}$, i.e. we need to check that, given $(x,y)\in D(K_q)$, we have:
\begin{equation}\label{eqPreserving}
\lambda_{i',j'}\circ \rho(x,y)=\rho\circ \lambda_{i',j'}(x,y).
\end{equation}

Because of the definition of Archdeacon embedding associated with $\omega_r$ and $\omega_c$, we have that
$$\rho\circ \lambda_{i',j'}(x,y)=\rho(\eta x,\eta y)=\rho_{\eta x}(\eta x,\eta y)=\rho_{\eta x}(\eta x,\eta x+(\eta y-\eta x)).$$
Here we have two cases:
\begin{equation}\label{3casi}\rho\circ \lambda_{i',j'}(x,y)= \begin{cases}
(\eta x,\eta x-\omega_r(\eta y-\eta x))
\mbox{ if } (\eta y-\eta x)\in \E(A);\\
(\eta x,\eta y+\omega_c(\eta x-\eta y))\mbox{ otherwise.}
\end{cases}\end{equation}

Then, by considering the row indexes modulo $m$ and the column indexes modulo $n$, and by recalling that $\eta=a_{i',j'}=\epsilon^{i'-1}\xi^{j'-1}$, we have
\begin{itemize}
\item $\eta a_{i,j}=a_{i'+i-1,j'+j-1}$;
\item $\omega_r(a_{i,j})= a_{i,j+1}$;
\item $\omega_c(a_{i,j})=a_{i+1,j}.$
\end{itemize}
Hence, if $z=a_{i,j}$,
$$\eta \omega_r(z)=\eta \omega_r(a_{i,j})= a_{i'+i-1, j'+j} =\omega_r(a_{i'+i-1,j'+j-1})=\omega_r(\eta z).$$
Reasoning similarly on the columns of $A$, we also have that:
$$\eta \omega_c(z)=\omega_c(\eta z).$$
We now notice that $(y-x)\in \E(A)$ if and only if $(\eta y-\eta x)\in \E(A)$. Therefore, for $(y-x)\in \E(A)$ we have 
$$\lambda_{i',j'}\circ \rho_x(x,x+(y-x))=(\eta x,\eta x-\eta \omega_r(y-x))=(\eta x,\eta x-\omega_r(\eta y-\eta x))$$
and, due to Equation \eqref{3casi},
$$(\eta x,\eta x-\omega_r(\eta y-\eta x))=\rho_{\eta x}(\eta x,\eta x+(\eta y-\eta x))=\rho\circ \lambda_{i',j'}(x,y).$$
Instead, if $(y-x)\not \in \E(A)$, again because of \eqref{3casi}, we derive that
$$\lambda_{i',j' }\circ \rho_x(x,x+(y-x))=(\eta x,\eta x+\eta \omega_c(x-y))=$$
$$(\eta x,\eta x-\omega_c(\eta x-\eta y))
=\rho_{\eta x}(\eta x,\eta x+(\eta y-\eta x))=\rho\circ \lambda_{i',j'}(x,y).$$

Hence, in both cases, Equation (\ref{eqPreserving}) is satisfied and $\lambda_{i',j'}\in Aut_0(\Pi_{m,n})$.

Now, since for every $\eta=a_{i',j'} = \epsilon^{i'-1}\xi^{j'-1}$ the associated graph automorphism $\lambda_{i',j'}$ is an automorphism of the embedding $\Pi_{m,n}$, the automorphism group $Aut_0(\Pi_{m,n})$ contains a subgroup isomorphic to $\Z_{m} \times \Z_{n} \cong \Z_{mn}$ as $m$ and $n$ are coprime.

Now we prove that this is exactly the group of the orientation-preserving automorphisms that fix zero. We first recall from Remark \ref{rem:group} that $Aut_0^+(\Pi_{m,n})$ is isomorphic to a subgroup of the cyclic group of order $2mn$. Moreover, from the definition of the $\Pi$ we deduce that for every $x \in \mathbb{F}_q^*$  the edge $\{0,x\}$ belongs to a face $F_1$, whose length is $m$, and to a face $F_2$, whose length is $n$. Since $m\not=n$, any element $\sigma \in Aut_0^+(\Pi_{m,n})$ must preserve the face-length. This means that $\sigma$ is of the form $\rho_0^{2\ell}$ for some $\ell \in [1,mn]$. Therefore we have at most $mn$ elements in $Aut_0^+(\Pi_{m,n})$.

It follows that $Aut_0^+(\Pi_{m,n})$ is isomorphic to $\Z_{mn}$.
\end{proof}

Now we will show that these embeddings do not admit orientation-reversing automorphisms. We begin by proving a technical lemma:
\begin{lem}\label{lem:fixzero}
Let $\Pi$ be an Archdeacon embedding of $K_{2mn+1}$ such that every edge is on a face whose boundary is an $m$-cycle and on a face whose boundary is a $n$-cycle, where $m\not=n$. Then, any $\sigma\in Aut_0^-(\Pi)$ fixes only the vertex $0$.
\end{lem}

\begin{proof}
Let $\Pi$, $m$ and $n$ be as in the statement, and choose any $\sigma\in Aut_0^-(\Pi)$. Set $\rho_0=(x_1,x_2,\dots,x_{2mn})$ to be the local rotation around $0$. Then, by Remark \ref{rem:group}, we have that $\sigma(x_j)=x_{\ell-j}$ for some $\ell \in \{1,\dots,2mn\}$.

Assume now that $\sigma $ fixes a vertex $x_i$ for some $i \in \{1,\dotsc, 2mn\}$. Hence, the edge $\{0,x_i\}$ is fixed by the action of $\sigma$. Let $F_1$ and $F_2$ be the faces of length $m$ and $n$ respectively containing $\{0,x_i\}$. Since $F_1$ and $F_2$ have different lengths, it follows that both faces are pointwise fixed by the action of $\sigma$. In particular $\sigma$ fixes $x_{i+1}$. On the other hand we have that $\sigma(x_i)=x_{\ell-i}=x_i$, that implies
$$\sigma(x_{i+1})=x_{\ell-(i+1)}=x_{i-1}\not=x_{i+1}$$
where the last relation holds because $2mn>3$.
It follows that $\sigma$ does not have any fixed point other than $0$.
\end{proof}

\begin{prop}\label{lem:only_plus}
Let $\Pi$ be an Archdeacon embedding of $K_{2mn+1}$ such that every edge is on a face whose boundary is an $m$-cycle and on a face whose boundary is a $n$-cycle, where $m$ and $n$ are distinct odd integers. Then $Aut_0^-(\Pi)$ is empty.
\end{prop}
\begin{proof}
Let $\Pi$, $m$ and $n$ be as in the statement, and let $\sigma$ be an automorphism in $Aut_0^-(\Pi)$. By Lemma \ref{lem:fixzero}, it follows that $\sigma$ has no fixed vertices other than $0$, hence by Remark \ref{rem:group} we have $\sigma(x_i)=x_{\ell-i}$ for some $\ell \in \{1,\dots,2mn\}$. Here $\ell$ must be odd, since otherwise we would have $i$ such that $\ell-i \equiv i \pmod{2mn}$, and hence $\sigma(x_i)=x_{\ell-i}=x_i$. Given an odd $\ell \in \{1,\dots,2mn\}$, the equation $\ell-i\equiv i+1\pmod{2mn}$ has two solutions, thus there exist exactly two indexes $j$ and $k$ such that $\sigma(x_{j}) = x_{j+1}$, $\sigma(x_{j+1}) = x_{j}$ and $\sigma(x_{k}) = x_{k+1}$, $\sigma(x_{k+1}) = x_{k}$.

Now we assume, without loss of generality, that $m>n$. We also assume that $n>3$ (the case where $n=3$ will be considered later).

Let $F_1$ be the face of length $m$ containing the edge $\{x_{j}, x_{j+1}\}$. Since $\sigma$ exchanges $x_{j}$ and $x_{j+1}$, and $F_1$ is the unique face of length $m$ containing these vertices, we have that $F_1$ is fixed by the action of $\sigma$. Now, as $m$ is odd, exactly one vertex of $F_1$ is fixed by $\sigma$, and from Lemma \ref{lem:fixzero} we deduce that this vertex is $0$, hence $0 \in F_1$.
Moreover, the vertices that are adjacent to $0$ in $F_1$ are exchanged by the action of $\sigma$, and, since $m>3$ and $F_1$ is simple, they must be $x_{k}$ and $x_{k+1}$. Similarly, if $F_2$ is the face of length $n>3$ containing the edge $\{x_{j}, x_{j+1}\}$, we obtain again that $0 \in F_2$ and that $x_{k}$ and $x_{k+1}$ are adjacent to $0$ in $F_2$ as well. We then gain a contradiction by noticing that $\rho_0(x_{k+1})\not=x_k$ and hence the path $(x_{k}, 0, x_{k+1})$ can not be contained in two different faces.

Finally, let us suppose $n=3$. In this case, since $m>n$, we still have that $(x_{k}, 0, x_{k+1})$ belongs to a face $F_1$ of length $m$. Here we have that, considering the face $F_2$ of length $n=3$ that contains the edge $\{x_k,x_{k+1}\}$, $\sigma$ must exchange $x_k$ and $x_{k+1}$, and thus it fixes the third point of $F_2$ that must be $0$. But this means that $(x_{k}, 0, x_{k+1})$ also belongs to a face $F_2$ of length $n=3$. Since this path can not be contained in two different faces we obtain a contradiction also in this case.

It follows that $Aut_0^-(\Pi)$ is empty.
\end{proof}
\begin{rem}\label{uppersize}
For every Archdeacon embedding $\Pi$ and integers $m,\ n$ that satisfy the hypothesis of Proposition \ref{lem:only_plus}, the following holds:
$$|Aut_0(\Pi)|=|Aut_0^+(\Pi)|\leq mn.$$
Indeed any automorphism that fixes zero is in $Aut_0^+(\Pi)$ and, since each edge belongs to two faces of different lengths, any element of $Aut_0^+(\Pi)$ is of the form $\rho_0^{2\ell}$ with $\ell \in \{1,\dots,mn\}$.
\end{rem}
We are then able to derive the following result, that exactly determines the size of the full automorphism group of this class of Archdeacon embedding.
\begin{thm}\label{thm:pos_cyclic2}
Let $q= 2mn+1$ be a prime power, with $m,n$ odd and coprime. Let $\omega_c$ and $\omega_r$ be the natural orderings from top to bottom and from left to right and let  $A_{m,n}$ be the array defined in Theorem \ref{B1}. Then, the Archdeacon embedding $\Pi_{m,n}$ induced by $\omega_r, \ \omega_c$ and $A_{m,n}$ is such that:
$$Aut_0(\Pi_{m,n}) \cong \Z_{mn}\mbox{ and } |Aut(\Pi_{m,n})|={2mn+1 \choose 2}.$$
Moreover, the faces of $\Pi_{m,n}$ are simple cycles of length $m$ and $n$.
\end{thm}

\begin{proof}
The statement follows from Theorem \ref{thm:pos_cyclic} and Lemma \ref{lem:only_plus}.
\end{proof}
\begin{rem}
This Theorem shows that an Archdeacon embedding $\Pi$ over $\mathbb{F}_q$ can have a group $Aut_0(\Pi)$ whose size reaches the upper bound $mn=\frac{q-1}{2}$ of Remark \ref{uppersize}, and we note that this upper bound can be chosen to be arbitrarily large.
Indeed, the embeddings considered here have been obtained using Heffter arrays over the group $\mathbb{F}_q$, but if $q=p$ is a simple prime, they can be obtained using classical Heffter arrays (over $\Z_p$). This happens infinitely many times since there are infinite primes $p$ of the form $p=2mn+1$ where $m,n$ are odd and coprime.

To prove this statement, it suffices to fix $n=3$ and look for primes $p$ of the form $6m+1$ with $m$ coprime with $3$, that is $m\equiv 1,2 \pmod{3}$. This is equivalent to look for $p \equiv 7,13 \pmod{18}$, and it is well known that there are infinitely many primes in these congruence classes.
\end{rem}
In a very recent paper, Buratti presented a construction of a 
rank-one $H(m,n)$ for a wide spectrum of prime powers $q=2mn+1$, which he plans to complete in a future paper. He found Heffter arrays with a large group of symmetries (\textit{multipliers}), and many of these arrays can be used to construct Archdeacon embeddings. Here, we focused on the case where $m$ and $n$ are odd and coprime because, under this assumption, the array has  group  of multipliers of maximum order (see Proposition 4.3 of \cite{B}), and we obtain an Archdeacon embedding whose automorphism group is the largest possible.

We conclude the paper by discussing the Archdeacon embedding (and its automorphisms) associated to the rank-one Heffter array of Example 3.3 of \cite{B}.
\begin{ex}
Consider the pair $(m,n) = (3,5)$. These two numbers are coprime and odd, and since $q = 2mn+1 = 31$ is  prime, we can apply Theorem \ref{thm:pos_cyclic} and construct the following $H(3,5)$:
\[
A=\begin{array}{|r|r|r|r|r|}\hline
1 & 2& 4 & 8 & 16 \\ \hline
5 & 10 & 20 & 9 & 18 \\ \hline
25 & 19 & 7 & 14 & 28 \\ \hline
\end{array}
\]
where $X = (1,2,4,8,16)$ and $Y =(1, 5, 25)$ are the subgroups of $\mathbb{Z}_{31}$ of order $5 $ and $3$, respectively. Since $3$ and $5$ are coprime, we can consider the natural orderings of each row from left to right, and of each column from top to bottom. From these orderings and the array $A$ we can then construct an Archdeacon embedding $\Pi_{3,5}$ of the complete graph $K_{31}$, whose vertices are identified with the elements of $\Z_{31}$.

Starting from the cell filled with the element $1$, we can write the rotation:
\[
\begin{aligned}
\rho_0 = \, &(1, -2, 10, -20, 7, -14, 8, -16, 18, -5, 25, -19, 2, -4, 20, -9, \\&14, -28, 16, -1, 5, -10, 19, -7, 4, -8, 9, -18, 28, -25).
\end{aligned}
\]
What can then be noticed is that $\rho_0$ is invariant under conjugation by $\lambda_{\eta}$ where $\lambda_{\eta}$ is the multiplication by an element $\eta\in \mathbb{Z}_{31}$ that is contained in $A$. Let us, for example, consider $\eta=9$ and note that $9\in A$ and let us define $\lambda_\eta(x) = \eta x$. Then, if $\rho_0 = (x_1,\dotsc, x_{2mn})$, with $x_1 = 1$, it holds:
\[
\begin{aligned}
\lambda_\eta \circ \rho_0 =\, & (1, -18, 4, -10, 16, -9, 2, -5, 8, -20)(-2, 28, -8, 19, -1, 14, -4, 25, -16, 7)\\& (5, -28, 20, -19, 18 , -14, 10, -25, 9, -7)= \rho_0\circ \lambda_\eta.
\end{aligned}
\]
That implies:
\[
\lambda_\eta \circ\rho_0(x) = \rho_0\circ \lambda_\eta(x).
\]
and hence we can see that $\lambda_{\eta}$ is an automorphism since
\[
\lambda_\eta \circ\rho(x,y) =\lambda_\eta \circ\rho(x,x+(y-x)) = (\eta x,\eta x+\eta\rho_0(y-x))=\]
\[
(\eta x,\eta x+\rho_0(\eta y-\eta x))=\rho\circ \lambda_\eta(x,y).
\]
The action of $\eta$ can be also seen directly on every cell of the array $A$. Let $a_{i,j}$ denote the element of $A$ in the $(i,j)$-th cell, and assume that $\eta = a_{i',j'}$ for some pair $(i',j')$. Then:
\[
\eta a_{i,j} = a_{i'+i-1, j'+j-1},
\]
where the row and column indexes are viewed modulo $m$ and $n$ respectively. Let then $\eta A_{3,5}$ denote the array whose $(i,j)$-th cell is filled by $\eta a_{i,j}$.

Then, for $\eta = a_{i',j'}$, and for any face $F$ given by, say, the $j$-th column of $A_{3,5}$, we have that $\eta F$ is one of the faces obtained from the $j$-th column of $\eta A_{3,5}$. Moreover, this is also a face $F'$ obtained from the $(j+j'-1)$-th column of $A_{3,5}$ consistently with the fact that $\lambda_\eta$ maps faces into faces.

Here, setting again $\eta = 9 = 5^1 2^3$, we can see that:
\[
\eta \cdot \begin{array}{|r|r|r|r|r|}\hline
1 & 2& 4 & 8 & 16 \\ \hline
5 & 10 & 20 & 9 & 18 \\ \hline
25 & 19 & 7 & 14 & 28 \\ \hline
\end{array} =
\begin{array}{|r|r|r|r|r|}\hline
9 & 18& 5 & 10 & 20 \\ \hline
14 & 28 & 25 & 19 & 7 \\ \hline
8 & 16 & 1 & 2 & 4 \\ \hline
\end{array}.
\]
Now, as an example, we verify that a face of $\Pi_{3,5}$ is mapped into another face by the action of $\eta=5^1 2^3$.
For instance, we pick $F_1 = (0,2,12)$ that is obtained from the second column of $A_{3,5}$. Then we consider $\eta F_1=(0,18,15)$, that is the development of the second column of $\eta A_{3,5}$. We conclude by recognising that $\eta F_1$ is also obtained by translating the development of the fifth column of $A_{3,5}$, indeed
$$\eta F_1=(0,18,15)=(16,3,0)+15 $$
and $(0,16,3)$ is the face obtained as the development of the fifth column.
\end{ex}

\section*{Acknowledgements}
The authors were partially supported by INdAM--GNSAGA.

\end{document}